\documentclass[12pt]{article}

\usepackage{amsmath,amssymb,amsthm}
\usepackage{graphicx,subfigure}
\usepackage[scriptsize]{caption}

\newtheorem{thm}{Theorem}
\newtheorem{cor}[thm]{Corollary}
\newtheorem{lem}[thm]{Lemma}
\newtheorem{prop}[thm]{Proposition}

\theoremstyle{remark}
\newtheorem{remark}[thm]{Remark}

\theoremstyle{definition}
\newtheorem{defn}[thm]{Definition}
\newtheorem*{defn*}{Definition}

\newenvironment{reftheorem}[1]{\begin{trivlist}\item[\hskip
\labelsep{\bf Theorem \ref{thm:#1}}]\it}
{\end{trivlist}}

\title{Intrinsically Linked Graphs in Projective Space}
\author{
		Joel Foisy
		\and
		Jason Bustamante
		\and
		Jared Federman
		\and
		Kenji Kozai
		\and
		Kevin Matthews
		\and
		Kristen McNamara
		\and
		Emily Stark
		\and
		Kirsten Trickey
}

\begin{document}

\maketitle

\begin{abstract}
We examine graphs that contain a non-trivial link in  
every embedding into real projective space, using a weaker notion of  
unlink than was used in \cite{flapan06}.  We call such graphs intrinsically
linked in $\mathbb{R}P^3$. We fully characterize such graphs with connectivity
0,1 and 2. We also show that only one Petersen-family graph is intrinsically
linked in $\mathbb{R}P^3$ and prove that $K_7$ minus any two edges is also
minor-minimal intrinsically linked. In all, 594 graphs are shown to be
minor-minimal intrinsically linked in $\mathbb{R}P^3$.
\end{abstract}

\section{Introduction}

We can represent knots in $\mathbb{R}P^3$ as closed curves or unions of arcs
in the closed 3-ball, $D^3$, such that the endpoints of the arcs lie on
$\partial D^3$. Because $\mathbb{R}P^3$ can be obtained from $D^3$ by
identifying antipodal points of $\partial D^3$, the set of endpoints of the
arcs must be symmetric over the origin. Fix an arbitrary great circle as the
equator. Using ambient isotopy, we can move the arcs so that all of the the
endpoints lie on the equator in general position. Then, the arcs can be
projected onto the disc bounded by the equator with over- and under-crossings
\cite{drobotukhina91,manturov04}.

Projective space has a non-trivial first homology  group,
$H_1(\mathbb{R}P^3) \cong \mathbb{Z} / 2\mathbb{Z}$. The generator for the
group, $g$, is the
cycle originating from the line in $D^3$ that runs between the north and south
poles.  Mroczkowski \cite{mroczkowski03} has shown that every knot in
$\mathbb{R}P^3$ can be transformed into either the trivial cycle or $g$ by
crossing changes and Reidemeister moves on an $\mathbb{R}P^2$ projection of the
knot. This suggests that there exist two non-equivalent unknots in
$\mathbb{R}P^3$. For the rest of the paper, we will refer to cycles that can be
``unknotted'' into a cycle homologous to $g$ as \textit{1-homologous cycles}
and cycles that can be ``unknotted'' into a null-homologous cycle as
\textit{0-homologous cycles}.

In $\mathbb{R}^3$, a two component link $L_1 \cup L_2$ is the unlink if and
only if $L_1$ and $L_2$ are both the unknot and there exist $A,B \subset
\mathbb{R}^3$, both homeomorphic to $B^3$, such that $A \cap B = \emptyset$,
$L_1 \subset A$, and $L_2 \subset B$. Because $g$ cannot be contained within
a sphere, using this definition in $\mathbb{R}P^3$ gives us a unique unlink
consisting of two 0-homologous unknots. However, a 0-homologous unknot and a
1-homologous unknot in $\mathbb{R}P^3$ may be drawn in a projection onto
$\mathbb{R}P^2$ with no crossings. On the other hand, two disjoint 1-homologus
unknots will always cross. Consequently, two reasonable definitions for
unlinks in $\mathbb{R}P^3$ exist.

Let $M$ be a 3-manifold.

\begin{defn} \label{defn:stronglyunlinked}
	Let $L_1 \cup L_2$ be a two-component link in $M$. If $L_1$
	and $L_2$ are both unknots and there exist $A,B \subset M$, both
	homeomorphic to $B^3$, such that $A \cap B = \emptyset$, $L_1 \subset A$,
	and $L_2\subset B$, then $L_1$ and $L_2$ are \textit{strongly unlinked}, and
	$L_1 \cup L_2$ is called the two-component unlink.
\end{defn}

\begin{defn} \label{defn:unlinked}
	Let $L_1 \cup L_2$ be a two-component link in $M$. If $L_1$ and $L_2$ are
	both unknots and there exists $A \subset M$ homeomorphic to $B^3$
	such that $L_1 \subset A$ and $L_2 \subset A^C$, then $L_1$ and $L_2$ are
	\textit{unlinked}, and $L_1 \cup L_2$ is a two-component unlink.
\end{defn}

Notice that definitions \ref{defn:stronglyunlinked} and \ref{defn:unlinked}
are equivalent when $M \cong \mathbb{R}^3$. Similarly, we can define
\textit{strongly splittable} and \textit{splittable} by removing the
condition that both components are unknots.

\begin{defn}
	Let $G$ be a graph. If every embedding of $G$ into $M$
	contains a pair of cycles that form a non-trivial two-component link,
	then $G$ is \textit{intrinsically linked in $M$}.
\end{defn}

Graphs that are intrinsically linked in $\mathbb{R}^3$ have been completely
classified through the work of Conway and Gordon \cite{conway83}, Sachs
\cite{sachs84}, and Roberston, Seymour, and Thomas \cite{robertson95}. They
have shown that a graph is intrinsically linked in $\mathbb{R}^3$ if and only
if it contains one of the Petersen-family graphs (the 7 graphs obtained from
$K_6$ by a sequence of $\triangle-Y$ and $Y-\triangle$ exchanges) as a minor.

Flapan, et al \cite{flapan06} classifies the set of all graphs that are
intrinsically linked when using Definition \ref{defn:stronglyunlinked}. The
complete minor-minimal set for intrinsic linking in any 3-manifold, $M$, is
the same as in $\mathbb{R}^3$ --- namely, the Petersen-family graphs --- when
the two-component unlink is defined to be the union of cycles which bound discs
that do not intersect. In $\mathbb{R}P^3$, their definition coincides with
Definition \ref{defn:stronglyunlinked}.

However, $K_6$ embeds in the projective plane, as shown in Figure \ref{fig:k6},
so there exists an embedding of $K_6$ into projective space for which every
two-component link is an unlink, as given by Definition \ref{defn:unlinked}.
Thus, with this definition, $K_6$ is not
intrinsically linked. For the remainder of this paper, unless otherwise noted,
trivial and non-trivial links will be defined using Definition
\ref{defn:unlinked}.

\begin{figure}
	\begin{center}
		\includegraphics[scale=0.5]{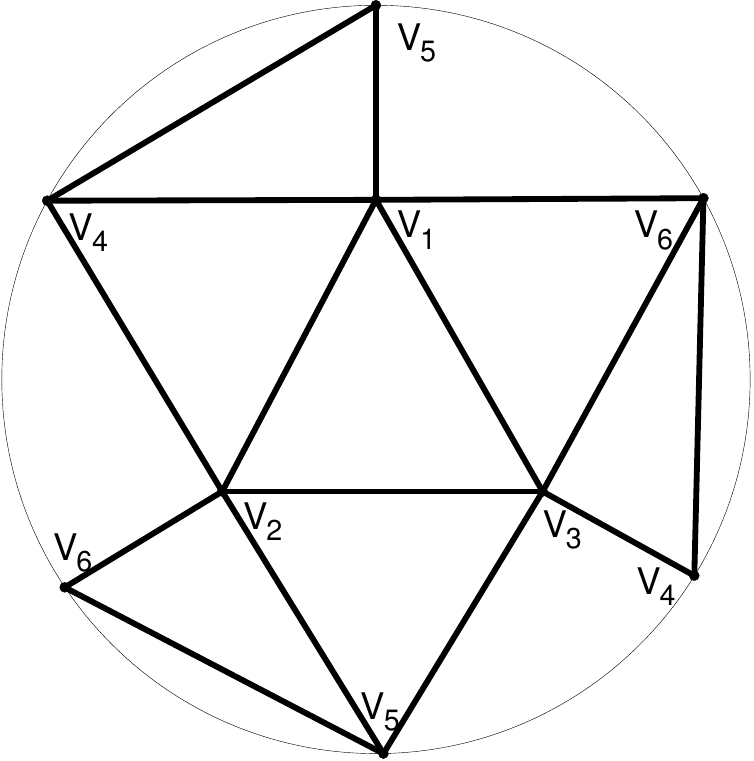}
		\caption{ An embedding of $K_6$ into $\mathbb{R}P^2$. The bounding circle
			is identified using the antipodal map to obtain $\mathbb{R}P^3$.}
		\label{fig:k6}
	\end{center}
\end{figure}

In this paper, we will prove the following theorems.

\begin{thm}\label{thm:lowconnectivity}
	Let $\mathcal{P}$ be the set of all Petersen-family graphs excluding
	the graph obtained from $K_{4,4}$ by removing an edge. Let $A,B,G$ be graphs
	such that $G$ has $k$-connectivity with vertex cut set $\{v_1,\dots,v_k\}$,
	$G = A \cup B$, and $V(A \cap B) = \{v_1,\dots,v_k\}$.
	\begin{enumerate}
		\item If $k=0$ or $1$, then $G$ is minor-minimal intrinsically linked in
			$\mathbb{R}P^3$ if and only if $A,B \in \mathcal{P}$.
		\item If $k=2$, then $G$ is minor-minimal intrinsically linked in
			$\mathbb{R}P^3$ if and only if $A',B' \in \mathcal{P}$,
			$E(A) = E(A') \setminus \{(v_1,v_2)\}$, and $E(B) = E(B') \setminus
			\{(v_1,v_2)\}$.
	\end{enumerate}
\end{thm}

The theorem classifies intrinsically linked graphs with low connectivity. The
first statement says that a graph that is disconnected (or with 1-connectivity)
is intrinsically linked if and only if it is the disjoint union (or union
along a vertex) of two Petersen-family graphs. The second statement is
analogous for graphs with 2-connectivity, but the edge between the two vertices
along which the Petersen-family graphs are joined is removed.

\begin{thm}\label{thm:k44-e}
	The graph obtained by removing an edge from $K_{4,4}$ is minor-minimal
	intrinsically linked in $\mathbb{R}P^3$.
\end{thm}

\begin{thm}\label{thm:k7-2e}
	The graphs obtained from $K_7$ by removing any two edges are minor-minimal
	intrinsically linked in $\mathbb{R}P^3$.
\end{thm}

\section{Definitions and Notation}

Before proceeding to our results, we begin with some elementary notation
and definitions.

\begin{defn}
	A graph $G = (V,E)$ is a set of vertices $V(G)$ and edges $E(G)$, where
	an edge is an unordered pair $(v_1,v_2)$ with $v_1,v_2 \in V$.
\end{defn}

\begin{defn}
	Let $G$ be a graph and $v_1, v_2, \dots, v_n \in V(G)$ and
	\begin{equation*}
		(v_1,v_2), (v_2,v_3),\dots, (v_{n-1},v_n), (v_n,v_1) \in E(G)
	\end{equation*}
	such that $v_i \neq v_j$ for $i \neq j$. Then, the sequences of vertices
	$v_1, v_2, \dots, v_n$ and edges $(v_1,v_2),(v_2,v_3),\dots,$ $(v_{n-1},v_n),
	(v_n,v_1)$ is an $n$-cycle in the graph $G$, denoted $v_1v_2\dots v_n$.
\end{defn}

In an abuse of notation, we will also refer to the image of a cycle
$v_1v_2\dots v_n$ in an embedding of the graph $G$ as the cycle
$v_1v_2\dots v_n$, when the distinction is clear.

The following notion of a graph minor allows us to specify when one graph
contains another graph within it.

\begin{defn}
	Let $G$ be a graph. If $H$ is a graph such that $H$ can be obtained from
	$G$ by a sequence of the following three operations:
	\begin{enumerate}
		\item removal of an edge \label{minor:edge}
		\item removal of a vertex \label{minor:vertex}
		\item contraction along an edge, \label{minor:contraction}
	\end{enumerate}
	then $H$ is called a \textit{minor} of $G$, written $H \leq G$. If $H \leq G$
	but $H \neq G$, then $H$ is called a \textit{proper minor} of $G$, written
	$H < G$.
	
	If $H \leq G$, we also call $G$ an \textit{expansion} of $H$.
\end{defn}

Nesetril and Thomas \cite{nesetril85} provide the following result for graph
minors in $\mathbb{R}^3$, and the general result in arbitrary 3-manifolds can be
proved by noticing that expansions preserve isotopy classes of cycles and links.

\begin{prop}[J. Nesetril and R. Thomas, 1985]
	Let $H$ be a graph that is intrinsically linked in a 3-manifold $M$. If $G$
	is a graph such that $H \leq G$, then $G$ is also intrinsically linked in
	$M$.
\end{prop}

\begin{defn}
	A graph $G$ is \textit{minor-minimal intrinsically linked in
	$M$} if $G$ is intrinsically linked in $M$ and no proper minor of $G$ is
	also intrinsically linked in $M$.
\end{defn}

In $\mathbb{R}^3$, the set of all minor-minimally intrinsically linked graphs
is given by the seven Petersen-family graphs. These graphs are obtained from
$K_6$ by $\triangle-Y$ and $Y-\triangle$ exchanges, where a $\triangle-Y$
exchange is the removal of three edges $(v_1,v_2),(v_1,v_3),(v_2,v_3)$ and
the addition of a vertex $v$ along with the edges $(v,v_1),(v,v_2),(v,v_3)$.
A $Y-\triangle$ exchange is the reverse operation.

As a result of Robertson and Seymour's proof of the Minor Theorem
\cite{robertson04}, the set of all minor-minimally intrinsically linked graphs
in $M$ is finite. This means that a full classification of minor-minimally
intrinsically linked graphs in $\mathbb{R}P^3$ is possible. Becase
projective space has a simple first homology group, it may not be unrealistic
to find a complete characterization for intrinsic linking.

\section{Linked Graphs with Low Connectivity}

Exactly six of the seven Petersen-family graphs have embeddings into
$\mathbb{R}P^2$ \cite{glover79,archdeacon83}, and thus have linkless
embeddings into $\mathbb{R}P^3$. We later show that the graph obtained by
removing an edge from $K_{4,4}$, which does not have a projective planar
embedding, is in fact intrinsically linked in $\mathbb{R}P^3$.

Although not all Petersen-family graphs are intrinsically linked in
$\mathbb{R}P^3$, we can use their intrinsic linking in $\mathbb{R}^3$ to
deduce some facts about embeddings with no non-trivial two-component links.

\begin{lem} \label{lemma:P-v}
	Let $P$ be a Petersen-family graph and $v$ be a vertex of $P$. If every cycle
	of $P \backslash \{v\}$ is 0-homologous in an embedding $f:P \rightarrow
	\mathbb{R}P^3$, then $f(P)$ contains a non-trivial link.
\end{lem}

\begin{proof}
	Let $L_1 \cup L_2$ be a link with a projection onto a disc representing
	$\mathbb{R}P^2$ such that
	$L_1$ is affine and does not cross the boundary of the projection and $L_2$
	is 1-homologous. Take a point $p$ of $L_2$ that intersects the boundary of
	the projection (the line at infinity). Let $U$ be a sufficiently small
	neighborhood of $p$ in the projection such that
	$L_1$ does not intersect $U$ and $L_2$ intersects $\partial U$ in exactly
	two points, $p'$ and $q'$. Connect $p'$ and $q'$ with a line segment $s$
	such that in the projection, $s$ crosses over every strand, and $s$ does not
	intersect the line at infinity. Define $L_2'$ as the cycle consisting of $s$
	and the segment of $L_2$ that is not in $U$. Then, $L_2'$ is a 0-homologous
	cycle such that the linking number of $L_1 \cup L_2$ is the same as the
	linking number of $L_1 \cup L_2'$ (see Figure \ref{fig:0h1h}).
	
\begin{figure}[htb]
	\begin{center}
		\subfigure[$p'$ and $q'$ in a small neighborhood of $p$]{
			\includegraphics[scale=0.7]{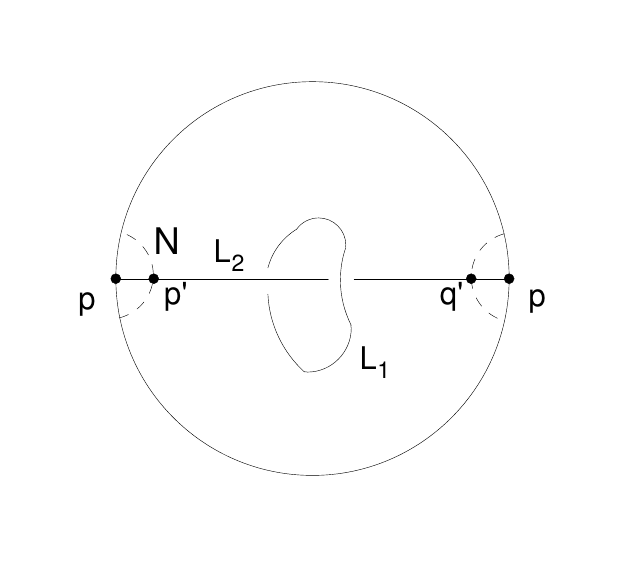}
		}
		\subfigure[$s$ crosses over all other arcs]{
			\includegraphics[scale=0.7]{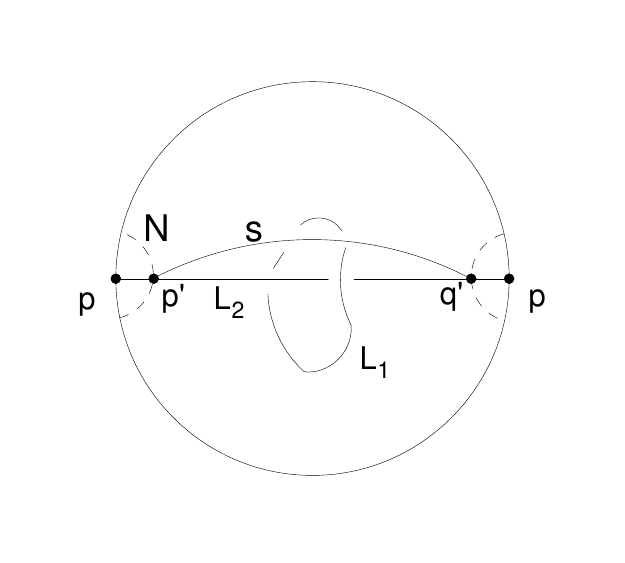}
		}
		
		\subfigure[$L_2'$ constructed from $L_2$]{
			\includegraphics[scale=0.7]{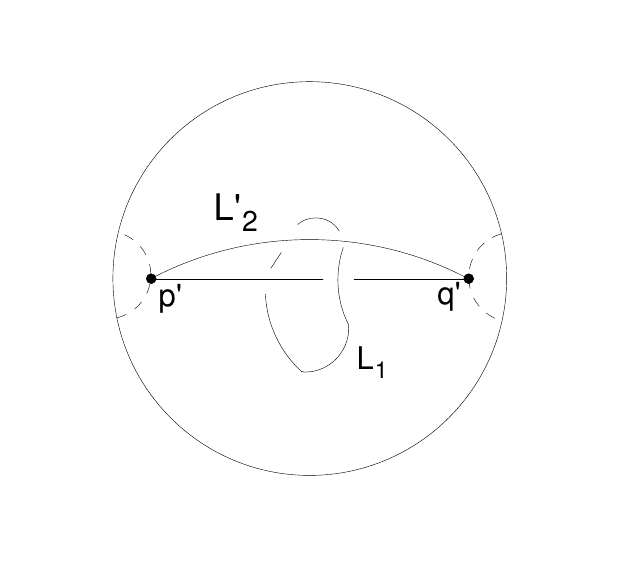}
		}
	\end{center}
	\caption{Conversion of a link consisting of an affine knot and 1-homologous
		knot into one consisting of two 0-homologous knots.}
	\label{fig:0h1h}
\end{figure}

	Consider $f(P)$. Using crossing changes and ambient isotopy,
	we may assume that the embedding for the subgraph $P \backslash \{v\}$ is
	affine so that $f(P \backslash \{v\})$ does not intersect the boundary
	of the projection (in other words, it does not pass through the line at
	infinity), $v$ lies on the boundary of the projection, and no point besides
	$v$ lies on the line at infinity.
	
	Define
	\begin{equation*}
		\lambda \equiv \sum_{\substack{L_1 \cup L_2\text{ is a}\\
			\text{two-component}\\ \text{link in } f(P)}} lk(L_1,L_2)\text{ (mod 2)},
	\end{equation*}
	where $lk(L_1,L_2)$ is the linking number of $L_1 \cup L_2$.
	The previous observation shows that there exists an affine embedding of $P$
	for which $\lambda$ is unchanged.
	Because crossing changes do not affect $\lambda$, the results of
	Conway and Gordon \cite{conway83} and Sachs \cite{sachs84} for $K_6$ and
	Petersen graphs in $\mathbb{R}^3$, respectively, imply that $\lambda
	\equiv 1 \text{ (mod 2)}$ for the embedding $f$ into $\mathbb{R}P^3$.
	Hence, the embedding must contain a two-component link with non-zero
	linking number, proving the lemma.
\end{proof}

Lemma \ref{lemma:P-v} allows us to completely classify intrinsically linked
graphs in $\mathbb{R}P^3$ with connectivity 0, 1, and 2, assuming that
$K_{4,4} \setminus \{e\}$ is intrinsically linked in $\mathbb{R}P^3$.

\begin{prop}
	Let $G=A \cup B$ be a 2-connected graph with vertex cut set 
	$V(A \cap B) = \{v_1,v_2\}$. Let $\overline{A} = A \cup \{(v_1,v_2)\}$ and
	$\overline{B} = B \cup \{(v_1,v_2)\}$. If $G$ is minor-minimal intrinsically
	linked in $\mathbb{R}P^3$, then $\overline{A}$ and $\overline{B}$ are
	intrinsically linked in $\mathbb{R}^3$.
\end{prop}

\begin{proof}
	Suppose $\overline{A}$ is not intrinsically linked in $\mathbb{R}^3$. Since
	$G$ is minor-minimal, $\overline{B} < G$ has a linkless embedding, $f$, in
	$\mathbb{R}P^3$. Let $g$ be an embedding of a closed 3-ball with interior
	$D$ into $\mathbb{R}P^3$ such that $f((v_1,v_2)) \subset g(\overline{D})$,
	only the vertices $v_1$ and $v_2$ intersect $\partial D$. and
	$f(\overline{B}\setminus \{(v_1,v_2)\})$ is in the complement of $g(D)$.
	Take a linkless
	embedding, $h$, of $\overline{A}$ in $\mathbb{R}^3 \cong D$. Then,
	$g \circ h$ is a linkless embedding of $\overline{A}$. Using ambient isotopy
	on $g \circ h$, we may assume that the arcs $f((v_1,v_2))$ and
	$g \circ h((v_1,v_2))$ coincide. The union of these two embeddings produces
	a linkless embedding of $G \cup (v_1,v_2)$ into $\mathbb{R}P^3$.
\end{proof}

\begin{prop}
	Let $G = (P_1 \cup P_2) \setminus \{(v_1,v_2)\}$ be a graph, where
	$P_1,P_2\in\mathcal{P}$ and $V(P_1 \cap P_2) = \{v_1,v_2\}$. Then
	$G$ is intrinsically linked in $\mathbb{R}P^3$.
\end{prop}

\begin{proof}
	Notice that both $P_1$ and $P_2$ are minors of $G$. Embed $G$ in
	$\mathbb{R}P^3$. By Lemma \ref{lemma:P-v},
	if $P_i$ does not contain any non-trivial links, then
	$P_i \backslash \{v_i\}$ must contain a 1-homologous cycle, for $i=1,2$.
	This results in two disjoint 1-homologous cycles. Hence, $G$ is linked.
\end{proof}

The previous two propositions prove Theorem \ref{thm:lowconnectivity} for
$k=2$, assuming Theorem \ref{thm:k44-e}. The results for $k=0$ and $k=1$ are
proved similarly, and Theorem \ref{thm:k44-e} is proved in the following
section.

For the case $k=0$, it is easy to see that there are $\left({6 \choose 2}\right)
=21$ minor-minimal intrinsically linked graphs in $\mathbb{R}P^3$. When
$k=1$, it is necessary to
count the different number of vertex classes in each graph to determine the
number of ways a pair of Petersen-family graphs may be glued along a vertex.
From Table \ref{tbl:vertexclasses}, the number of minor-minimal intrinsically
linked graphs with 1-connectivity in $\mathbb{R}P^3$ is determined to be 91.

\begin{table}
	\begin{center}
		\begin{tabular}{cc}
			\hline
			Graph &  Vertex Classes\\
			\hline
			$K_6$ & 1\\
			$K_{3,3,1}$ & 2\\
			$P_7$ & 3\\
			$P_8$ & 4\\
			$P_9$ & 2\\
			Petersen & 1\\
			\hline
		\end{tabular}
		\caption{Petersen-family graphs and the number of vertices, up to
			equivalence under graph isomorphism.}
		\label{tbl:vertexclasses}
	\end{center}
\end{table}

Define the \textit{vertex flipping number (VFN)} for some vertex pair
$\{x_1, x_2\}$ as
\begin{equation*}
	VFN(x_1, x_2) =
		\begin{cases}
			0 & \text{if $x_1 ~ x_2$}\\
			1 & \text{otherwise}
		\end{cases}
\end{equation*}
where $x_1 ~ x_2$ is equivalence under a graph isomorphism. Counting the
number of minor-minimal intrinsically linked graphs in Theorem
\ref{thm:lowconnectivity} when $k=2$ requires attention to the $VFN$ of
vertex pair classes, where two pairs of vertices are equivalent if there is
a graph isomophism taking one pair to the other. For each pair
$\{x_1,x_2\} \subseteq E(G_1), \{y_1,y_2\}
\subseteq E(G_2)$ of vertex pair classes for two graphs $G_1, G_2$, the number
of ways to glue $G_1$ and $G_2$ along the specified vertex pairs is
$VFN(x_1,x_2) VFN(y_1,y_2) + 1$. Table \ref{tbl:pairclasses} lists the number
of vertex pair classes of each type, and number of minor-minimal intrinsically
linked graphs in $\mathbb{R}P^3$ is 469.

\begin{table}
	\begin{center}
		\begin{tabular}{|c|c|c|c|}
			\hline
			 &  \multicolumn{3}{c|}{Vertex Pair Classes}\\
			Graph & Total & $VFN = 0$ & $VFN = 1$\\
			\hline
			$K_6$ & 1 & 1 & 0\\
			$K_{3,3,1}$ & 3 & 2 & 1\\
			$P_7$ & 5 & 2 & 3\\
			$P_8$ & 10 & 3 & 7\\
			$P_9$ & 6 & 4 & 2\\
			Petersen & 2 & 2 & 0\\
			\hline
		\end{tabular}
		\caption{Petersen-family graphs and the number of vertex pairs, up to
			equivalence under graph isomorphism.}
		\label{tbl:pairclasses}
	\end{center}
\end{table}

\section{$K_{4,4}$ With an Edge Removed}

In this section, we prove that the graph obtained by removing an edge from
$K_{4,4}$ is intrinsically linked in $\mathbb{R}P^3$.

We will need the following observation.

\begin{prop} \label{prop:k4}
	For every embedding into $\mathbb{R}P^3$, $K_{3,2}$ has an even number of
	1-homologous 4-cycles.
\end{prop}

\begin{proof}
	Whenever two cycles $C_1$ and $C_2$ intersect along an arc, $D$, we can
	define the sum of $C_1$ and $C_2$ to be $C_1 \cup C_2 \setminus D$. Then,
	the result can be obtained by noting that the sum of two 0-homo\-logous
	cycles and the sum of two 1-homologous cycles are 0-homologous cycles, and
	the sum of a 0-homologous cycle with a 1-homologous cycle is 1-homologous.
\end{proof}

\begin{lem} \label{thm:k33}
	If a graph $G$ isomorphic to $K_{3,3}$ is embedded in $\mathbb{R}P^3$ such
	that at least one of its cycles is 1-homologous, then the homology classes of
	all of the 4-cycles in the embedding of $G$ have one of two possibilities:
	\begin{enumerate}
		\item A cycle is 1-homologous if and only if it passes through a specified
			edge, $(u,v)$, of the graph. We call $(u,v)$ the
			\textit{including edge} and the homology pattern of the embedding a
			\textit{4-pattern}.
		\item A cycle is 1-homologous if and only if it does not pass through 
			two of the edges in $F \subset {E(G)}$, where $F$ is a specified set
			of three mutually disjoint edges of $G$. We call $F$ the set of
			\textit{excluding edges} and the homology pattern of the embedding 
			a \textit{6-pattern}.
	\end{enumerate}
\end{lem}
  
\begin{proof}
	Let $\{a_1,a_2,a_3\} \subset V(G)$ and $\{b_1,b_2,b_3\} \subset V(G)$ be the
	partition sets of $G$. Suppose $G$ contains a 1-homologous cycle. Then, it
	must contain a 1-homologous 4-cycle $C_1$. Let $H$ be a subgraph of $G$
	isomorphic to $K_{3,2}$ that contains $C_1$. By Proposition \ref{prop:k4},
	$H$ must contain two 1-homologous 4-cycles. Without loss of generality, they
	are the cycles $a_1b_1a_2b_2$ and $a_1b_1a_2b_3$. It also must
	be the case that the cycle $a_1b_2a_2b_3$ is 0-homologous.
	
	Now, consider the subgraph induced by $\{a_1, a_2, a_3, b_1, b_2\}$. By
	Proposition \ref{prop:k4}, one
	of the two cycles $a_1b_1a_3b_2$ and $a_2b_1a_3b_2$ is
	1-homologous, and the other is 0-homologous. Since interchanging $a_1$ and
	$a_2$ does not affect the choices made up to this point, without loss of
	generality, the cycle $a_1b_1a_3b_2$ is 1-homologous and the cycle
	$a_2b_1a_3b_2$ is 0-homologous.
	
	Next, consider the subgraph induced by $\{a_1,a_3,b_1,b_2,b_3\}$.
	Since the cycle $a_1b_1a_3b_2$ is
	1-homologous, then either the cycle $a_1b_1a_3b_3$ is also
	1-homologous and the cycle
	$a_1b_2a_3b_3$ is 0-homologous, or the cycle $a_1b_1a_3b_3$ is 0-homologous
	and the cycle $a_1b_2a_3b_3$ is 1-homo\-logous.
	
	\noindent\textbf{Case 1:} Cycle $a_1b_1a_3b_3$ is 1-homologous and
	cycle $a_1b_2a_3b_3$ is 0-homolo\-gous.
	
	Applying Proposition \ref{prop:k4} to all of the other subgraphs of $G$
	isomorphic to $K_{3,2}$ forces the last two cycles, $a_2b_1a_3b_3$
	and $a_2b_2a_3b_3$, to be 0-homologous. Observe that a cycle in $G$ is
	1-homologous if and only if it includes the edge $(a_1,b_1)$. Hence, this
	embedding of $G$ has a 4-pattern, with $(a_1,b_1)$ as its including edge.
	
	\noindent\textbf{Case 2:} The cycle $a_1b_1a_3b_3$ is 0-homologous and
	the cycle $a_1b_2a_3b_3$ is 1-homolo\-gous.
	
	Again, by using Proposition \ref{prop:k4} on the remaining $K_{3,2}$
	subgraphs of
	$G$, the cycles $a_2b_1a_3b_3$ and $a_2b_2a_3b_3$ must be
	1-homologous. A 4-cycle of $G$ is 0-homologous if and only if it contains
	two edges from the
	set $F = \{(a_1,b_3), (a_2,b_2), (a_3,b_1)\}$. The set $F$ is the set of
	excluding edges, and the embedding is a 6-pattern.
\end{proof}
  
\begin{reftheorem}{k44-e}
	The graph $G$ obtained by removing an edge from $K_{4,4}$ is minor-minimal
	intrinsically linked in $\mathbb{R}P^3$.
\end{reftheorem}

\begin{proof}
	Consider an embedding of $G = K_{4,4} \backslash \{(a_1,b_1)\}$, where
	\begin{equation*}
		\{a_1,a_2,a_3,a_4\}, \{b_1,b_2,b_3,b_4\} \subset V(G)
	\end{equation*}
	are the partition sets.
	
	Let $A$ be the subgraph induced by $\{a_2,a_3,a_4,b_2,b_3,b_4\}$, $B$
	be the subgraph induced by $\{a_1,a_2,a_3,b_2,b_3,b_4\}$, and $C$ be the
	subgraph induced by $\{a_2,a_3,a_4,b_1,b_2,b_3\}$. By Lemma \ref{thm:k33},
	$A$ contains no 1-homologous cycles, is a 4-pattern, or is a 6-pattern.
	
	\noindent\textbf{Case 1:} The subgraph $A$ contains no 1-homologous cycles.
	
	By Lemma \ref{lemma:P-v}, if the embedding is linkless, the subgraph induced
	by $\{a_1,a_2,a_3,a_4,b_2,b_3,b_4\}$ must contain a
	1-homologous cycle. Because $A$ does not contain any 1-homologous cycles, all
	such cycles must pass through $a_1$. Consider the subgraph induced by
	$\{a_1,a_2,a_3,b_2,b_3,b_4\}$. This $K_{3,3}$ subgraph must then have a
	4-pattern. Without loss of generality, the including edge is $(a_1,b_2)$.
	
	Similarly, the subgraph induced by $\{a_2,a_3,a_4,b_1,b_2,b_3\}$ contains a
	4\-pattern with including edge $(a_2,b_1)$. Then, $a_1b_2a_2b_3$ and
	$b_1a_2b_2a_3$ are disjoint 1-homolo\-gous cycles.
	
	\noindent\textbf{Case 2:} The subgraph $A$ contains a 4-pattern.
	
	Without loss of generality, $A$ has $(a_4,b_4)$ as its including edge.
	
	\noindent\textit{Subcase 2.1:} Either $B$ or $C$ has a 6-pattern.

	The subgraph $B$ cannot have a 6-pattern
	as then subgraph induced by $\{a_2,a_3,b_2,b_3,b_4\}$ would contain a
	1-homologous cycle,
	contradicting that all 1-homologous cycles in $A$ pass through
	its including edge.

	\noindent\textit{Subcase 2.2:} Both $B$ and $C$ contain no 1-homologous
	cycles.
	
	It is easy
	to see that all 1-homologous cycles of $G$ must pass through the edge
	$(a_4,b_4)$ by looking at the other four $K_{3,3}$
	subgraphs of $G$ and noticing that each subgraph must have a 1-homologous
	cycle by the including edge in $A$. If any subgraph of $G$ (not including
	$B$ and $C$) has a 6-pattern or a 4-pattern with an including edge that is
	not $(a_4,b_4)$, then this would force a 1-homologous cycle in $B$ or $C$.
	By Lemma \ref{lemma:P-v}, since all 1-homologous cycle pass through $a_4$,
	$G$ is linked.

	\noindent\textit{Subcase 2.3:} Both $B$ and $C$ have 4-patterns.

	If $B$ contains a 4-pattern, then its including edge must pass through
	$a_1$. Otherwise, $A$ contains a 1-homologous cycle disjoint from its
	including edge. Similarly, if $C$ contains a 4-pattern, then its including
	edge must pass through $b_1$. The subgraph $B$ has its including
	edge passing through $a_1$ and $C$ has its including edge passing through
	$b_1$. So we can find disjoint 1-homologous cycles in $G$.

	\noindent\textit{Subcase 2.4:} One of $B$ or $C$ has a 4-pattern and the
	other contains no 1-homologous cycles.

	Without loss of generality, assume that $B$ has a 4-pattern and $C$ contains
	no 1-homologous cycles. By the previous subcase, the including edge in
	$B$ has $a_1$ as an endpoint. We claim that the subgraphs induced by
	$\{a_2,a_3,a_4,b_1,b_2,b_4\}$ and
	$\{a_2,a_3,a_4,b_1,b_3,b_4\}$ must have 4-patterns: both contain
	1-homologous cycles due to $A$ having a 4-pattern, and if either contained a
	6-pattern, there would be a 1-homologous cycle in $C$. Any edge with $b_1$
	as an endpoint cannot be an including edge for these two graphs, since then
	$C$ would contain a 1-homologus cycle. Consequently,
	both subgraphs must have $(a_4,b_4)$ as its including edge. Otherwise,
	there would be a 1-homologous 4-cycle in $A$ that does not have $(a_4,b_4)$
	as one of its edges.
	
	If the including edge in $B$ does not have $b_4$ as its other endpoint,
	because the subgraph induced by $\{a_2,a_3,a_4,b_1,b_2,b_4\}$ has
	$(a_4,b_4)$ as its including
	edge, $G$ contains disjoint 1-homologous links. Otherwise, since the cycles
	$a_ib_2a_4b_4$ and $a_ib_3a_4b_4$ are 1-homologous from $A$ and cycles
	$a_1b_2a_ib_4$ and $a_1b_3a_ib_4$ are 1-homologous
	from $B$, then
	the subgraph induced by $\{a_1,a_i,a_4,b_2,\linebreak b_3,b_4\}$ has a
	4-pattern with
	$(a_i,b_4)$ as its including edge, for $i=2,3$. In this case, we have shown
	that all 1-homologous
	cycles pass through $b_4$, so by Lemma \ref{lemma:P-v}, $G$ is linked.
	
	\noindent\textbf{Case 3:} The subgraph $A$ has a 6-pattern.
	
	Without loss of generality, the
	excluding edges in $A$ are $(a_i,b_i)$ for $i=2,3,4$.
	Then, every $K_{3,3}$ subgraph of $G$ shares a $K_{3,2}$ with $A$, so it must
	contain a 1-homologous cycle.

	\noindent\textit{Subcase 3.1:} Both $B$ and $C$ contain 4-patterns.
	
	If $B$ contains a 4-pattern, its including edge must pass through $b_4$.
	Otherwise, $B$ contains a 1-homologous cycle from the 6-pattern in $A$ that
	does not pass through its
	including edge. Since the subgraph induced by $\{a_2,a_3,b_2,b_3,\linebreak 
	b_4\}$
	contains a 1-homologous cycle by $A$, then $B$ has it including edge passing
	through $a_2$ or $a_3$. Let $(a_i,b_4)$ be the including edge in $B$.
	
	Likewise, if $C$ has a 4-pattern, its including edge must be $(a_4,b_j)$,
	where $j=2$ or $3$. Then, it is easy to see that $G$ contains disjoint
	1-homologous cycles. If $C$ has a 6-pattern, then let $k=2,3$, $k \neq i$.
	Then, the subgraph induced by $\{a_k,a_4,b_1,b_2,b_3\}$ contains a
	1-homologous 4-cycle, one of which must pass through $b_1$. The 4-cycle
	that is disjoint from this cycle is also 1-homologous by the including edge
	in $B$, so $G$ is linked.

	\noindent\textit{Subcase 3.2:} Either $B$ or $C$ contain a 6-pattern.

	Without loss of generality, assume that $B$ has a 6-pattern. One of its
	excluding edges must be $(a_1,b_4)$
	since cycle $a_2b_2a_3b_3$ is 0-homologous by $A$, and $(a_1,b_4)$ is the
	only edge in $B$ that is disjoint from this cycle.
	Note that if $(a_2,b_2)$ and $(a_3,b_3)$
	are also excluding edges, then all cycles in the subgraph induced by
	$\{a_1,a_2,a_4,b_2,b_3,b_4\}$
	through $(a_2,b_3)$ are 1-homologous. We saw in the Subcase 2.1 that
	when there
	is a 4-pattern in a $K_{3,3}$ that is one adjacent (differs by one vertex)
	to a $K_{3,3}$ with a 6-pattern, then the graph is linked. Otherwise,
	$(a_2,b_3)$ and $(a_3,b_2)$ are the other excluding edges in $B$.
	
	Similarly, if $G$ does not contain any non-trivial links, then $C$ must have
	$(a_4,b_1)$, $(a_2,b_3)$, and $(a_3,b_2)$ as
	excluding edges. Hence, $a_1b_4a_2b_2$ and 
	$a_4b_1a_3b_3$ are disjoint 1-homologous cycles. So $G$ is linked.
	
	The graph $G$ is minor-minimal since any proper minor of $G$ embeds in the
	projective plane \cite{glover79,archdeacon83}.
\end{proof}

\section{$K_7$ Minus Two Edges}

We now prove that any graph obtained by removing two edges from $K_7$ is
minor-minimal intrinsically linked in $\mathbb{R}P^3$. There are two cases of
Theorem \ref{thm:k7-2e}: when the two edges are adjacent and when the two
edges are non-adjacent. We will use the following lemma.

\begin{lem}\label{k4lem}
	Given a linkless embedding of $K_6$, no $K_4$ subgraph can have all
	0-homologous cycles.
\end{lem}

\begin{proof}
Consider an embedding of $K_6$ for which there is a $K_4$ subgraph with all cycles 0-homologous.  By using crossing changes and ambient isotopy, this $K_4$ subgraph can be deformed so that it does not touch the line at infinity, and so that there are no crossings on it in a projection.  Denote the vertices of this $K_4$ by $\{v_1,v_2,v_3,v_4\}$ and denote the vertices not in the $K_4$ by $v_5$ and $v_6$.  One can deform the edge $(v_5,v_6)$ so that it is contained in the line at infinity, so that $v_6$ is placed at 12 o'clock and 6 o'clock, and so that $v_5$ is placed at 3 o'clock and 9 o'clock.  We may assume the edge $(v_5,v_6)$ goes from 12 o'clock to 3 o'clock (see Figure \ref {setup}).  

\begin{figure}[htb]
	\begin{center}
		\includegraphics[scale=1]{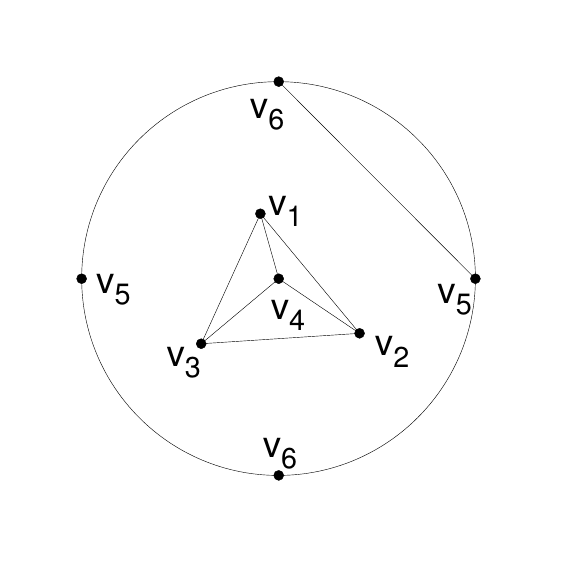}
	\end{center}
	\caption{We may deform the embedded graph to be in this position (not all edges are shown.)}
	\label{setup}
\end{figure}

Now, we claim that the edges connecting $v_6$ to the $K_4$ can be deformed
(using crossing changes and ambient isotopy) so that they are straight lines in
the projection that connect to the $K_4$ either from 12 o'clock, or from 6
o'clock.  We may assume that the edge connecting to $v_4$ is under all of the
other edges of the $K_4$ in the projection.  We will justify the claim for the
edge $(v_6,v_1)$.  Consider the embedded cycle formed the two (additional)
edges $e_1$ and $e_2$, where $e_1$ connects $v_1$ to the 12 o'clock $v_6$, and
$e_2$ connects $v_1$ to the 6 o'clock $v_6$, where both $e_1$ and $e_2$ are
straight edges in the projection.  This cycle is 1-homologous.  The edge
$(v_1,v_6)$ from the $K_6$ embedding breaks up this cycle into two cycles, one
formed by $e_1$ and $(v_1,v_6)$ and the other formed by $e_2$ and $(v_1,v_6)$.
One of these two cycles must be 1-homologous, and the other must be
0-homologous.  If the cycle formed by $e_1$ and $(v_1,v_6)$ is 0-homologous,
then $(v_1,v_6)$ can be deformed, using crossing change and ambient isotopy, to
$e_1$.  Similarly, if the cycle formed by $e_2$ and $(v_1,v_6)$ is 0-homologous,
then $(v_1,v_6)$ can be deformed to $e_2$.  This established our claim.  It
similarly follows that the edges connecting $v_5$ to the $K_4$ can be deformed
(using crossing changes and ambient isotopy) so that they are straight lines in
the projection that connect to the $K_4$ from either 3 o'clock or 9 o'clock.

Now, it cannot be the case that all of the edges connecting $v_6$ to the $K_4$
are incident to 12 o'clock, for then $v_1, v_2, v_3, v_4$ and $v_6$ would
induce a $K_5$ with all cycles 0-homologous, which cannot occur in a linkless
embedding of $K_6$ by Lemma \ref{lemma:P-v}.  Similarly, all of the edges cannot
be incident 6 o'clock, nor can all of the edges emanating from $K_5$ be incident
to 3 o'clock, nor can they all be incident to 9 o'clock.  Thus, there must be
exactly 1, 2 or 3 edges from the $K_4$ incident to 12 o'clock, and exactly 1,
2, or 3 edges from the $K_4$ incident to 3 o'clock.  In all cases but one,
there are a pair of disjoint 1-homologous cycles.  These disjoint 1-homologous
cycles would have been present in the original embedding of $K_6$.  For example,
if only $(v_1,v_6)$ is incident to 12 o'clock, and only $(v_2,v_5)$ is incident
to 3 o'clock, then $(v_1,v_6,v_3)$ and $(v_2,v_5,v_4)$ form disjoint
1-homologous cycles.

The only case that does not lead to disjoint 1-homologous cycles is the case when exactly 1 edge from $K_4$ is incident to 12 o'clock (6 o'clock) and exactly 1 edge from $K_4$ is incident to 3 o'clock (9 o'clock), and these two edges are incident to the same vertex in the $K_4$.  By symmetry, we may assume the edge $(v_1,v_6)$ is incident to 12 o'clock, and the edge $(v_1,v_5)$ is incident to 3 o'clock. Then, all 1-homologous cycles pass through $v_1$, so by Lemma \ref{lemma:P-v}, this embedding is linked.

\end{proof}

An easy consequence of Lemma \ref{k4lem} is as follows:

\begin{cor}\label{cor:2k3111}
	The graph on $9$ vertices obtained by pasting together two copies of 
	$K_{3,1,1,1}$ along the three vertices that are mutually non-adjacent
	is intrinsically linked in ${\mathbb R}P^3$.
\end{cor}

\begin{thm}
 The graph obtained from $K_7$ by removing two edges incident to a common
 vertex is minor-minimal intrinsically linked in ${\mathbb R}P^3$.
\end{thm}

\begin{proof}
Let $G$ be the graph obtained from $K_7$ by removing two edges incident to a
common vertex.
Let $v_1,v_2,.., v_7$ denote the vertices of $G$, with $v_7$ connected only to $v_1,v_2,v_3$ and $v_4$.  Embed $G$.  By the previous result, if the embedding is linkless, the $K_4$ induced on $\{v_1,v_2,v_3,v_4\}$ must contain a 1-homologous 3-cycle.  By a homology argument, then there must be a 1-homologous 3-cycle through the vertex $v_7$.  Without loss of generality, we may assume the 3-cycle is $(v_1,v_2,v_7)$.  If the embedding is linkless, then by the previous result, the $K_4$ induced by $\{v_3,v_4,v_5,v_6\}$ must contain a 1-homologous cycle, but this forces two disjoint 1-homologous cycles.  Thus, the embedding cannot be linkless.

The graph $G$ is minor-minimal since any proper minor of $G$ embeds in the projective plane \cite{glover79}\cite{archdeacon83}.
\end{proof}

\begin{thm}
The graph obtained from $K_7$ by removing two non-adjacent edges is
minor-minimal intrinsically linked in ${\mathbb R}P^3$.
\end{thm}

\begin{proof}
Let $K$ be the graph obtained from $K_7$ by removing two non-adjacent edges.

Label the vertices of $K_7$ as $\{v_1, v_2, ..., v_7\}$, and suppose edges
$(v_4,v_5)$ and $(v_6, v_7)$ are removed to result in the graph $K$.  Embed
$K$, and suppose the embedding is linkless.  We claim  that the $4$-cycle
$(v_4, v_7,v_5,v_6)$ cannot be 1-homologous.  If it were, then a cycle of the
form $(v_1,v_i,v_j)$ is 1-homologous, for $i, j \in \{4,5,6,7\}$, with
$i \neq j$.  Without loss of generality, suppose $(v_1,v_5,v_7)$ is
1-homologous; then the subgraph induced by $\{v_2,v_3,v_4,v_6\}$ forms $K_4$,
and since $K$ contracts onto $K_6$, by Lemma \ref {k4lem}, there must be a
disjoint 1-homologous cycle, which is a contradiction.  Similarly, the
following $3$-cycles must also be 0-homologous:  $(v_i, v_4,v_7)$,
$(v_j, v_5,v_7)$, $(v_k, v_5,v_6)$, and $(v_m, v_4,v_6)$, where
$i,j,k,m \in \{1,2,3\}$.  It follows that for the subgraph induced by the
vertices $\{v_1, v_4, v_7, v_5, v_6\}$, every cycle is 0-homologous.  Since this
subgraph contracts onto $K_4$, and since $K$ contracts onto $K_6$ (by
contracting the same edge), it follows from Lemma \ref{k4lem} that there must
be non-splittable links in the embedding.  It follows that $K$ is intrinsically
linked in ${\mathbb R}P^3$.

The graph $K$ is minor-minimal for intrinsic linking since any proper minor of
$K$ is projective planar, as shown by Glover, et al \cite{glover79} and
Archdeacon \cite{archdeacon83}, or does not contain any intrinsically
$\mathbb{R}^3$-linked graphs as a minor, so there exists a linkless embedding
of every proper minor into a 3-ball.

The graph $K$ is minor-minimal for intrinisic linking in ${\mathbb R}P^3$
since any proper minor of $K$ is either projective planar
\cite{archdeacon83,glover79} or becomes ($\mathbb{R}^2$) planar after the
removal of a vertex (and hence is not intrinsically linked in space). In
either case no minor is intrinisically linked in ${\mathbb R}P^3$.

\end{proof}

\section{Other Intrinsically Linked Graphs}

It is not too hard to see that $\triangle-Y$ exchanges preserve intrinsic
linking as in $\mathbb{R}^3$ \cite{motwani88}, so any graph generated
from a known intrinsically $\mathbb{R}P^3$-linked graph by a sequence
of $\triangle-Y$ exchanges is also intrinsically linked. Corollary
\ref{cor:2k3111} provides a graph with several $\triangle$ subgraphs.

Notice that two copies of $K_{3,1,1,1}$ glued along the three mutually
non-adjacent vertices is the same as gluing two copies of $K_6$ along
three vertices $v_1,v_2,v_3$ and then removing the triangle composed of
the three edges
between $v_1,v_2,$ and $v_3$. For notational convenience, each
copy of $K_6$ with the triangle removed will be referred to as
$K_6 \therefore$, and $K_6 \therefore K_6$ denotes the gluing of
two copies of $K_6 \therefore$ along the three vertices that are
mutually non-adjacent. By Corollary \ref{cor:2k3111}, $K_6 \therefore K_6$ is
intrinsically linked in $\mathbb{R}P^3$.

In general, we define $G \therefore$ to be a graph with three
marked vertices which are mutually non-adjacent. If
$G_1 \therefore$ and $G_2 \therefore$ are two such graphs,
then $G_1 \therefore G_2$ is a graph obtained by gluing the
two graphs along the three marked vertices of $G_1 \therefore$
and $G_2 \therefore$. The resulting graph
may not be unique if permutation of the marked vertices does not
yield a graph isomorphism of each $G_i \therefore, i=1,2$. In such
cases, we will differentiate
between the (up to) three distinct graphs by a subscript, as in
$G_1 \therefore_1 G_2$, $G_1 \therefore_2 G_2$, and
$G_1 \therefore_3 G_2$.

\begin{prop}
	There are 18 intrinsically linked graphs in $\mathbb{R}P^3$ with
	3-connectivity that can be obtained from $K_6 \therefore K_6$ by
	$\triangle - Y$ exchanges. 
\end{prop}

\begin{proof}
	Figure \ref{fig:triy} shows $K_6 \therefore$ with the marked vertices
	($d,e,$ and $f$)
	shown as open circles, and the edges which were removed from
	$K_6$ shown as dotted lines. All other edges are not shown. In
	the subsequent figures, only edges added by $\triangle - Y$
	exchange are shown.

	\begin{figure}
		\begin{center}
			\includegraphics[scale=0.78]{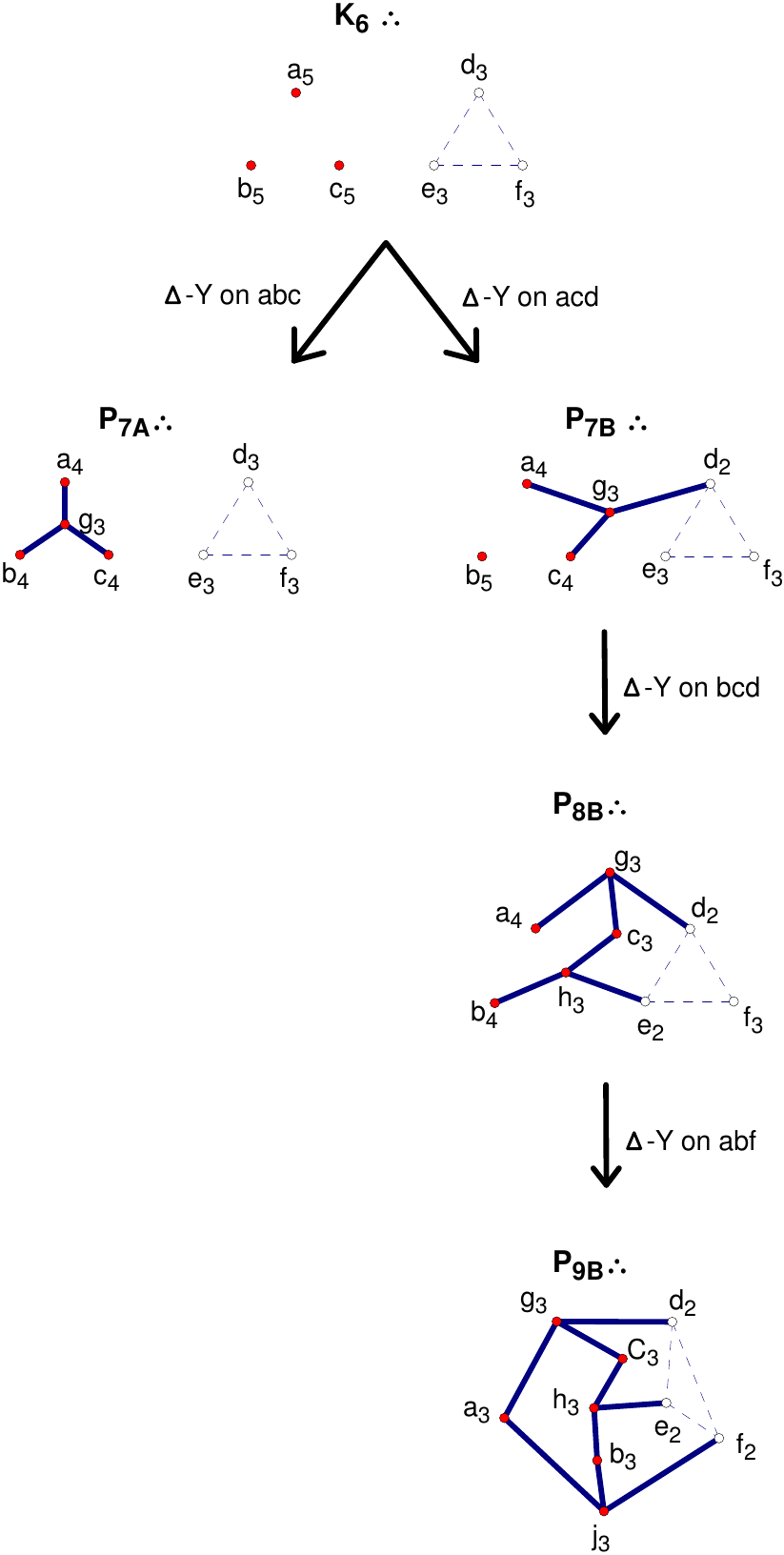}
			\caption{Graphs obtained from performing $\triangle - Y$
				exchanges on $K_6 \therefore$ and each subsequent graph,
				with subscripts denoting degree of the vertex, the open
				circles representing the three marked vertices, and dashed
				edges representing edges removed from the original $K_6$.}
			\label{fig:triy}
		\end{center}
	\end{figure}

	The graphs obtained by repeated $\triangle - Y$ exchanges on
	$K_6 \therefore$ are shown in Figure \ref{fig:triy}. Consequently,
	there are 18 intrinsically linked graphs in $\mathbb{R}P^3$
	obtained from $K_6 \therefore K_6$. The graphs $P_{7B} \therefore
	P_{7B}$, $P_{7B} \therefore P_{8B}$, and $P_{8B} \therefore
	P_{8B}$ each have two different configurations. The configuration
	where the $Y$ subgraphs of the two copies of $P_{7B} \therefore$
	are glued together along a shared vertex
	is $P_{7B} \therefore_1 P_{7B}$, and the configuration where they
	are not is $P_{7B} \therefore_2 P_{7B}$. Similarly, if the $Y$ in
	$P_{7B} \therefore$ shares a vertex with a $Y$ in $P_{8B} \therefore$,
	we have $P_{7B} \therefore_1 P_{8B}$, and otherwise, we have
	$P_{7B} \therefore_2 P_{8B}$. If the $Y$ subgraphs are paired up
	in $P_{8B} \therefore P_{8B}$, we have $P_{8B} \therefore_1 P_{8B}$,
	and otherwise, we have $P_{8B} \therefore_2 P_{8B}$.
\end{proof}

\begin{remark}
	The graphs $P_{7A} \therefore K_6$, $P_{7A} \therefore P_{7A}$,
	$P_{7A} \therefore P_{7B}$, $P_{7A} \therefore P_{8B}$, and
	$P_{7A} \therefore P_{9B}$ are intrinsically linked in
	$\mathbb{R}P^3$, but all contain
	$K_{4,4}$ with an edge removed as a minor. Hence, they are not
	minor-minimal intrinsically linked in $\mathbb{R}P^3$.
\end{remark}

To show that the remaining 13 intrinsically linked graphs in
$\mathbb{R}P^3$ are minor-minimal, we will use the following result
from \cite{brouwer07,ozawa07}.

\begin{thm}
	Let $P$ be a property preserved under $\triangle - Y$ exchange.
	Let $G$ be a graph that contains at least one degree three vertex
	and is minor-minimal with respect to $P$. Let $G'$ be a graph
	obtained from $G$ by a $Y - \triangle$ exchange. If $G'$ has
	property $P$, then $G'$ is also minor-minimal with respect to $P$. 
\end{thm}

Thus, we need only show that no proper minor of
$P_{9B} \therefore P_{9B}$ is intrinsically linked in
$\mathbb{R}P^3$.

\begin{thm}
	The 13 graphs obtained from $\triangle-Y$ exchange on
	$K_6 \therefore K_6$ which do not contain $P_{7A}\therefore$
	as a subgraph are minor-minimal intrinsically linked in
	$\mathbb{R}P^3$.
\end{thm}

\begin{proof}
We will show that no proper minor of $P_{9B}\therefore P_{9B}$ is 
intrinsically linked in $\mathbb{R}P^3$.

\begin{figure}
	\begin{center}
		\includegraphics[scale=0.4]{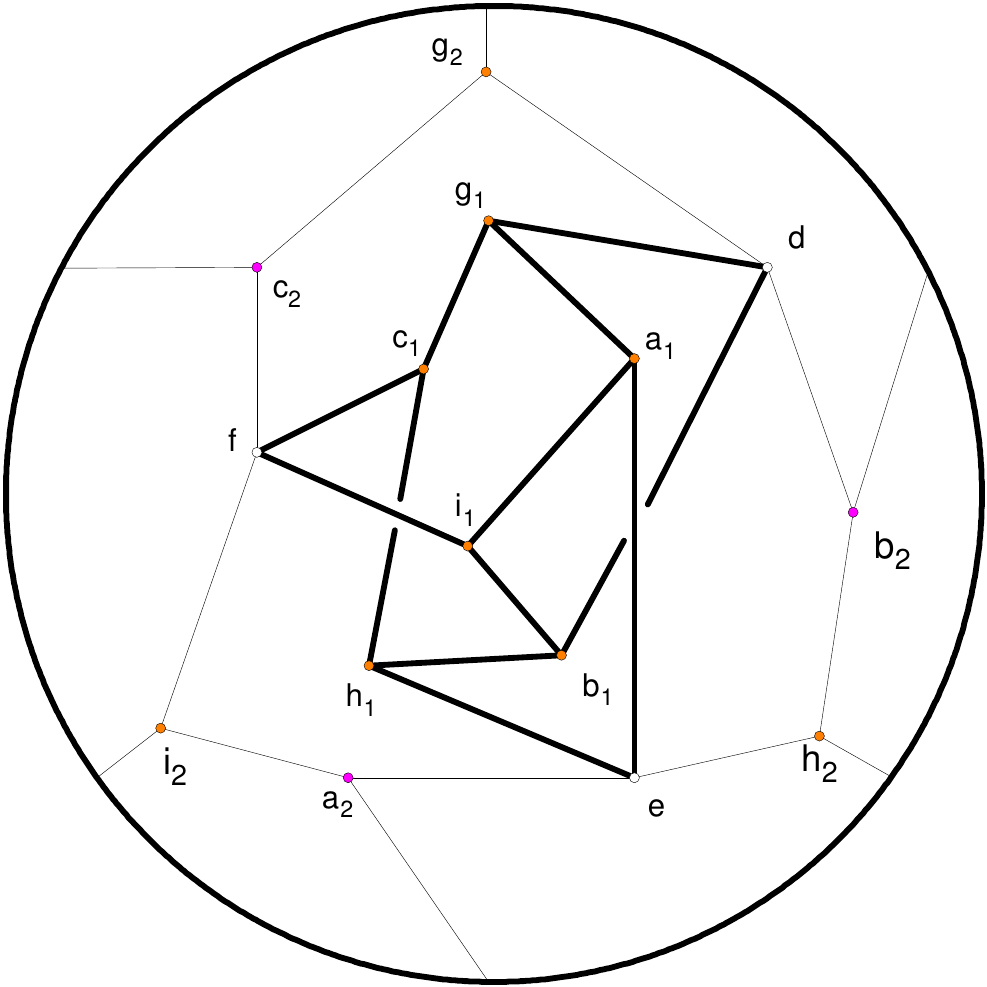}
		\caption{$P_{9B} \therefore P_{9B}$ connected along the marked vertices
			$d,e,f$. One copy of $P_{9B} \therefore$ is in bold.}
		\label{fig:P9BP9B}
	\end{center}
\end{figure}

Consider $P_{9B}\therefore P_{9B}$ as drawn in Figure \ref{fig:P9BP9B}.
The only pair of linked cycles in this embedding is $c_1h_1ea_1g_1$ and
$fi_1b_1db_2i_2$. There are two vertex equivalence classes in
$P_{9B}\therefore P_{9B}$: $\{d,e,f\}$ and
$V(P_{9B}\therefore P_{9B})\setminus \{d,e,f\}$.

To check for minor-minimality, it suffices to show that removing or contracting
any edge in $P_{9B} \therefore P_{9B}$ results in graph that is not
intrinsically linked in $\mathbb{R}P^3$.

There are two edge classes (up to graph isomorphism) that need to be considered.
Removing edge $(a_1,e)$ or edge $(c_1,h_1)$ from the embedding in Figure
\ref{fig:P9BP9B} results in a linkless embedding. Contracting edge $(f,c_1)$
or edge $(a_1,i_1)$ in Figure \ref{fig:P9BP9B} results in a linkless embedding
since the edge contractions send vertices on the (only) two linked cycles to
the same point, thus eliminating the non-trivial link.
\end{proof}

\section{Remarks}

Using the weaker definition of  unlinked components in Definition
\ref{defn:unlinked} allows the use of 1-homologous cycles to reduce the number
of crossings in a graph embedding in projective space. Thus, intrinsically
linked graphs in $\mathbb{R}P^3$ are more complex. Unlike in $\mathbb{R}^3$,
where there are simple arguments showing that there are no minor-minimal
intrinsically linked graphs with connectivity 0, 1, or 2, such graphs exist
in projective space. Using careful combinatorics, one can show that there are
21 disconnected graphs, 91 graphs with 1-connectivity, and 469 graphs with
2-connectivity which are minor-minimal intrinsically linked in $\mathbb{R}P^3$.
It is not too hard to see that $\triangle-Y$ exchanges preserve intrinsic
linking as in $\mathbb{R}^3$, so we predict that there are many more
minor-minimal intrinsically linked graphs than the ones we have observed in
this paper. In particular, the graphs obtained by removing two edges from
$K_7$ have a myriad of triangles in which we can perform a $\triangle-Y$
exchange, leading to more intrinsically linked graphs. Some of these graphs
have $K_{4,4}$ with an edge removed as a minor, but others have yet
to be explored fuly. It would be of interest
to see which of these are in fact minor-minimal intrinsically linked in
$\mathbb{R}P^3$.

\section{Acknowledgments}
This material is based upon work obtained by research groups at the 2007
and 2008 Research Experience for Undergraduates Program at SUNY Potsdam
and Clarskon University, advised by Joel Foisy and supported by the National
Science Foundation under Grant No. 0646847 and the National Security
Administration under Grant No. 42652.

\bibliography{rp3linkedgraphs}

\end{document}